\documentclass[a4paper,oneside,11pt, reqno, psamsfonts]{amsproc}
\usepackage{graphicx}
\usepackage{float}
\usepackage{pstricks}
\usepackage{amscd}
\usepackage{amsmath}
\usepackage{amsxtra}
\usepackage[T1]{fontenc}
\usepackage[utf8]{inputenc}
\usepackage{lipsum}
\usepackage{tikz}
\usetikzlibrary{arrows}
\usetikzlibrary{calc}

\usepackage{amsfonts,amssymb,amsmath,amsthm}
\usepackage{url}
\usepackage{enumerate}

\newcommand{\strutstretchdef}{\newcommand{\strutstretch}{\vphantom{\raisebox{1pt}{$\big($}\raisebox{-1pt}{$\big($}}}}
\theoremstyle{plain}
\newtheorem{theorem}{Theorem}[section]

\newtheorem{lemma}[theorem]{Lemma}

\newtheorem{corollary}[theorem]{Corollary}

\theoremstyle{definition}
\newtheorem{definition}[theorem]{Definition}

\newtheorem{example}[theorem]{Example}
\newtheorem{question}[theorem]{Question}
\newtheorem*{assumption}{Assumption}
\theoremstyle{remark}
\newtheorem{remark}[theorem]{Remark}

\numberwithin{equation}{section}

\newlength{\struh}
\newlength{\textminustop}
\strutstretchdef
\newcommand{\tr}{\textrm}

\newcommand{\ran}{\tr{ran }}

\newcommand*{\childm}[3]{{\mathsf{Chi}}^{\langle#1\rangle}_{\mathscr #2}(#3)}

\newcommand*{\lambdab}{\boldsymbol\lambda}

\newcommand*{\parent}[1]{\mathsf{par}(#1)}

\newcommand*{\rootb}{{\mathsf{root}}}

\usepackage{mathrsfs}
\usepackage{epsfig}
\usepackage[active]{srcltx}

\newcommand{\ncom}{\newcommand}
\ncom{\bq}{\begin{equation}}
\ncom{\eq}{\end{equation}}
\ncom{\beqn}{\begin{eqnarray*}}
\ncom{\eeqn}{\end{eqnarray*}}
\ncom{\beq}{\begin{eqnarray}}
\ncom{\eeq}{\end{eqnarray}}
\ncom{\nno}{\nonumber}
\ncom{\rar}{\rightarrow}
\ncom{\Rar}{\Rightarrow}
\ncom{\noin}{\noindent}
\ncom{\bc}{\begin{centre}}
\ncom{\ec}{\end{centre}}
\ncom{\sz}{\scriptsize}
\ncom{\rf}{\ref}
\ncom{\sgm}{\sigma}
\ncom{\Sgm}{\Sigma}
\ncom{\dt}{\delta}
\ncom{\Dt}{Delta}
\ncom{\lmd}{\lambda}
\ncom{\Lmd}{\Lambda}
\ncom{\eps}{\epsilon}
\ncom{\pcc}{\stackrel{P}{>}}
\ncom{\dist}{{\rm\,dist}}
\ncom{\sspan}{{\rm\,span}}
\ncom{\im}{{\rm Im\,}}
\ncom{\sgn}{{\rm sgn\,}}
\ncom{\ba}{\begin{array}}
\ncom{\ea}{\end{array}}
\ncom{\eop}{\hfill{{\rule{2.5mm}{2.5mm}}}}
\ncom{\eoe}{\hfill{{\rule{1.5mm}{1.5mm}}}}
\ncom{\eof}{\hfill{{\rule{1.5mm}{1.5mm}}}}
\ncom{\hone}{\mbox{\hspace{1em}}}
\ncom{\htwo}{\mbox{\hspace{2em}}}
\ncom{\hthree}{\mbox{\hspace{3em}}}
\ncom{\hfour}{\mbox{\hspace{4em}}}
\ncom{\hsev}{\mbox{\hspace{7em}}}
\ncom{\vone}{\vskip 2ex}
\ncom{\vtwo}{\vskip 4ex}
\ncom{\vonee}{\vskip 1.5ex}
\ncom{\vthree}{\vskip 6ex}
\ncom{\vfour}{\vspace*{8ex}}
\ncom{\norm}{\|\;\;\|}
\ncom{\integ}[4]{\int_{#1}^{#2}\,{#3}\,d{#4}}
\ncom{\inp}[2]{\langle{#1},\,{#2} \rangle}
\ncom{\Inp}[2]{\Langle{#1},\,{#2} \Langle}
\ncom{\vspan}[1]{{{\rm\,span}\#1 \}}}
\ncom{\dm}[1]{\displaystyle {#1}}

\def \N{\mathbb{N}}

\begin{document}
\title[Analytic $m$-isometries]{Analytic $m$-isometries without the wandering subspace property}

\author[A. Anand, S. Chavan and S. Trivedi]{Akash Anand, Sameer Chavan and Shailesh Trivedi}

%\thanks{The second named author was partially supported by NSF Grant DMS-1302666.} \

\address{Department of Mathematics and Statistics\\
Indian Institute of Technology Kanpur, India}
\email{akasha@iitk.ac.in}
   \email{chavan@iitk.ac.in}
 \email{shailtr@iitk.ac.in}

%\dedicatory{Dedicated to the memory of Serguei Shimorin}

\thanks{The work of the third author is supported through the Inspire Faculty Fellowship DST/INSPIRE/04/2018/000338}



\begin{abstract}
The wandering subspace problem for an analytic norm-increasing $m$-isometry $T$ on a Hilbert space $\mathcal H$ asks whether every $T$-invariant subspace of $\mathcal H$ can be generated by a wandering subspace. An affirmative solution to this problem for $m=1$ is ascribed to Beurling-Lax-Halmos, while that for $m=2$ is due to Richter.
In this paper,
we capitalize on the idea of weighted shift on one-circuit directed graph to construct a family of analytic cyclic $3$-isometries, which do not admit the wandering subspace property and which are norm-increasing on the orthogonal complement of a one-dimensional space. Further, on this one dimensional space, their norms can be made arbitrarily close to $1$. We also show that if the wandering subspace property fails for an analytic norm-increasing $m$-isometry, then it fails miserably in the sense that the smallest $T$-invariant subspace generated by the wandering subspace is of infinite codimension.
\end{abstract}

\maketitle

\section{Introduction}

The structural theory of $z$-invariant subspaces  in the Hilbert
spaces of analytic functions led many mathematicians to develop the
theory of non-isometric Hilbert space operators. For instance, the
structural theory of $z$-invariant subspaces in Dirichlet space led
Richter  \cite{R} to study the  class of $2$-isometries. A similar
quest inspired Aleman  \cite{Al} to study the class of
Dirichlet-type operators. Further, the structural theory of
$z$-invariant subspaces in the Bergman space motivated Shimorin \cite{S}
to study the class of Bergman-type operators. The basic problem here
is to see whether or not the given analytic operator
admits the wandering subspace property.
In this terminology (attributed to Halmos), a celebrated result of
Beurling \cite{B} says that the unweighted shift operator
in the Hardy space of the unit disc admits the wandering subspace property. A counter-part of
Beurling's Theorem for the Bergman shift was a
long-standing open problem. Indeed, a similar result had been deemed virtually impossible by many analysts in view of
the huge lattice of its invariant subspaces (see \cite[Corollaries 3.3 and 3.4]{ABFP}). Finally in
\cite{ARS}, the trio Aleman-Richter-Sundberg settled this problem
affirmatively. 
 Later in the influential paper \cite{S},
Shimorin not only obtained an alternate proof of their theorem but at
the same time developed an axiomatic approach to the wandering
subspace problem (see, for instance, \cite[Section 6.3]{H-K-Z}).

All the Hilbert spaces occurring below are complex,
infinite-dimensional and separable.
Let $\mathcal H$ denote a Hilbert space and
let $\mathcal B(\mathcal H)$ denote the unital $C^*$-algebra of bounded linear operators on $\mathcal H$. The Hilbert space adjoint of $T \in \mathcal B(\mathcal H)$ is denoted by $T^*$. The kernel of $T$ is denoted by $\ker T,$ whereas the range of $T$ is denoted by $\mbox{ran}\, T$.
An operator $T \in \mathcal B(\mathcal H)$ is {\it left-invertible} if there exists $L \in \mathcal B(\mathcal H)$ (a {\it left-inverse}) such that $LT=I,$ where $I$ denotes the identity operator on $\mathcal H$.
If $T$ is left-invertible, then $T^*T$ is invertible. This fact provides a canonical choice of left-inverse $L:=T'^*$ for any left-invertible operator $T$, where $T':=T(T^*T)^{-1}.$ The operator $T'$ is referred to as the {\it Cauchy dual} of $T,$ a notion coined and studied by Shimorin \cite{S}.
For an operator $T \in \mathcal B(\mathcal H)$, 
the {\it hyper-range} of $T$ is given by $$T^{\infty}(\mathcal H) := \bigcap_{n=0}^{\infty}T^n\mathcal H.$$
If $T$ is left-invertible, then $T^{\infty}(\mathcal H)$ is a closed subspace of $\mathcal H.$
An operator $T \in \mathcal B(\mathcal H)$ is {\it analytic} if $T^{\infty}(\mathcal H) = \{0\}$.
Note that an analytic operator on a non-zero Hilbert space is never invertible.
We say that $T \in \mathcal B(\mathcal H)$ admits the {\it wandering subspace property} if $[\ker T^*]_T=\mathcal H,$ where
\beq
\label{WS}
[\ker T^*]_T := \bigvee \{T^nh : n \in \mathbb N, ~h \in \ker T^* \},
\eeq
where $\mathbb N$ denotes the set of non-negative integers.
Here $\ker T^*$ is the {\it wandering subspace} in the sense that $$\ker T^* \perp T^n (\ker T^*), \quad n \geqslant 1.$$
Following \cite{S}, we say that a left-invertible operator $T \in \mathcal B(\mathcal H)$ admits {\it Wold-type decomposition} if $$T = U \oplus A ~~\mbox{on} ~~{\mathcal H} = {\mathcal H}_u \oplus {\mathcal H}_a,$$ where $ {\mathcal H}_u$ is the hyper-range of $T,$
$U$ is unitary on  ${\mathcal H}_u$, $A$ is analytic on ${\mathcal H}_a,$ and $A$ admits the wandering subspace property. It turns out that for a left-invertible operator $T \in \mathcal B(\mathcal H)$, $T$ is analytic if and only if the Cauchy dual $T'$ of $T$ admits wandering subspace property (see \cite[Proposition 2.7]{S}). In particular, $T$ admits Wold-type decomposition if and only if $T'$ admits Wold-type decomposition.

 Given a positive integer $m$ and $T \in \mathcal B(\mathcal H)$, set
   \begin{align*}
B_m(T) := \sum_{k=0}^m (-1)^k {m \choose k}{T^*}^kT^k.
   \end{align*}
   We say that 
%$T$ is {\it contractive} (resp. {\it expansive}) if $B_1(T) \geqslant 0$ (resp. $B_1(T) \leqslant 0$). 
an operator $T$ is {\it norm-increasing} or {\it expansive} (resp. {\it isometry}) if $B_1(T) \leqslant 0$ (resp. $B_1(T)=0$).
For $m \geq 2$, an operator
$T \in \mathcal B(\mathcal H)$ is said to be an {\it $m$-isometry} (resp. {\it $m$-concave}) if $B_m(T) = 0$ (resp. $(-1)^mB_m(T) \leqslant 0$).
A $2$-concave operator is usually known as {\it concave} operator.
It turns out that the approximate point spectrum of any $m$-isometric operator is contained in the unit circle and hence its spectrum is contained in the closed unit disc (see \cite[Lemma 1.21]{Ag-St}). For the basic theory of $m$-isometries, the reader is referred to \cite{Ag-St}. 
\begin{remark} \label{ap-sp}
Let $T \in \mathcal B(\mathcal H)$ be an $m$-concave operator. It follows from the definition that the approximate point spectrum of $T$ is contained in the unit circle. Since the approximate point spectrum contains the boundary of the spectrum, the spectrum of $T$ is contained in the closed unit disc.   In particular, the Cauchy dual operator of $T$ is well-defined. 
\end{remark}
%while for the study of operators Cauchy dual to $m$-isometries, refer to \cite{AC} and \cite{ACJS}.

This paper is partly motivated by the following problem addressed by Shimorin in \cite[Pg 185]{S}:
\begin{question} \label{Q}
If $T \in \mathcal B(\mathcal H)$ is a norm-increasing $m$-isometry (or norm-increasing $m$-concave), then whether $T$ admits Wold-type decomposition?
\end{question}

%\noindent
The above question has affirmative answer in cases $m=1$ (isometries) \cite{B, L, H} and $m=2$ (concave operators) \cite{R} (cf. \cite[Corollary 2.1]{O}). In case 
$m > 2$, the best known result till date, due to Shimorin, is as follows (cf. \cite[Corollary 1.6]{Se}):

\begin{theorem}\cite[Theorem 3.8]{S} \label{Shimorin}
Assume that $T \in \mathcal B(\mathcal H)$ is norm-increasing and satisfies the inequality 
\beq \label{w-2-c}
 T^{*2}T^2 - 3 T^*T  + 3 I - T'^*T' - P_{\ker T^*} \leqslant 0,
 \eeq
 where $P_{\ker T^*}$ denotes the orthogonal projection of $\mathcal H$ onto $\ker T^*.$
Then $T$ is a $3$-concave operator, which admits Wold-type decomposition. 
\end{theorem}

Unlike the condition of $3$-concavity, the condition \eqref{w-2-c} is not stable with respect to the process of taking the restriction to an invariant subspace. This is one of the obstructions in deriving the wandering subspace property for norm-increasing operators satisfying \eqref{w-2-c}.
Further, as pointed out in \cite{S}, the main difficulty in Question \ref{Q} is whether the analyticity implies the wandering subspace property. To see this deduction, note that $\mathcal H_u:=T^{\infty}(\mathcal H)$ is $T$-invariant and $T|_{\mathcal H_u}$ is invertible. Since the restriction of a norm-increasing $m$-isometry
to an invariant subspace is again a norm-increasing $m$-isometry, by \cite[Proposition 3.4]{S},  $T|_{\mathcal H_u}$ is unitary and $\mathcal H_u$ is reducing for $T.$ 
%The same remark holds for any $(2m)$-concave expansion. 
The assumption that $T$ is norm-increasing in Question \ref{Q} is essential since there exist invertible (and hence non-analytic) cyclic $3$-isometries on an infinite-dimensional Hilbert space, which are not unitary (refer to \cite[Section 4]{Ag-St}).
However, we are not aware of any known example of an analytic $m$-isometry, which does not admit the wandering subspace property. 
One of the purposes of this paper is to settle the wandering subspace problem for analytic $m$-isometries in the negative. 
%In particular, our solution answers Question \ref{Q} to a large extent by
This is achieved by constructing a family of analytic cyclic $3$-isometries, which do not admit the wandering subspace property (the reader is referred to \cite[Corollary 1]{KLS}, \cite[Pg 186]{S}, \cite[Proposition 2]{HZ}, \cite[Corollary 6.2]{CDS}, \cite[Theorem 2.2]{Ca}, \cite[Corollary 4.1]{NRW}, \cite[Theorem 1.5]{Se} for a variety of results pertaining to the phenomenon of failure of wandering subspace property in Hardy, Bergman and Dirichlet-type spaces). Further, these $3$-isometries are norm-increasing except on a one-dimensional space, where their norms can be made arbitrarily close to $1$. Needless to say,
these examples shed light on the role of the expansivity property in Question \ref{Q}.
Our construction capitalizes on the idea of weighted shift operator on directed graph recently boosted in the realm of graph-theoretic operator theory (refer to \cite{JJS}, \cite{CT}, \cite{BJJS}, \cite{DPP}). 

Here is the outline of the paper. In Section 2, we present some preparatory results including a characterization of a class of analytic weighted shifts on a locally finite one-circuit directed graph (see Theorem \ref{main-1}). Section 3 includes the construction of analytic $3$-isometries without the wandering subspace property  (see Example \ref{construction}). In Section 4, we prove that for analytic norm-increasing $m$-concave operators, either the wandering subspace property holds or it fails miserably.
This result applies to the class of analytic weighted shifts on a locally finite one-circuit directed graph as discussed in Section 2.
In the remaining part, we collect preliminaries related to weighted shift on directed graphs mostly from \cite{JJS} and \cite{BJJS}.

\subsection{Weighted shifts on one-circuit directed graphs}

The notion of adjacency operator on a directed graph first appeared in \cite{FSW}, while that of weighted shift on a (one-circuit) directed graph, central to the present investigations, has been formulated and investigated in \cite{BJJS} (refer to \cite[Section 3]{BJJS} for its connection with weighted composition operators 
on discrete measure spaces). The reader is referred to \cite[Chapter 2]{JJS} for an excellent exposition on directed graphs and all relevant notions including weighted shift on directed trees (see, in particular, the discussion on the parent $\parent{\cdot}$ as a partial function). It is worth mentioning that $\parent{\cdot}$ becomes a function in case the underlying directed graph is a one-circuit directed graph.
\begin{definition}
Let $\mathscr T = (V, E)$ be a rooted directed tree with root $\rootb$. For $l \in \mathbb N$, let $C_l:=\emptyset$ if $l=0$ and $C_l:=\{w_1, \ldots, w_l\}$ otherwise. A {\it one-circuit directed graph associated with the rooted directed tree} $\mathscr T = (V, E)$ is a directed graph $\mathscr G = (W, F)$ given by 
\beqn
W &:= &V \sqcup C_l, \\
F &:=& \begin{cases} E \cup \{(\rootb, \rootb)\} & \mbox{if~}l=0, \\
E \cup \{(\rootb, w_1), (w_1, w_2), \ldots, (w_{l-1}, w_l), (w_l, \rootb)\} & \mbox{otherwise}.
\end{cases}
\eeqn
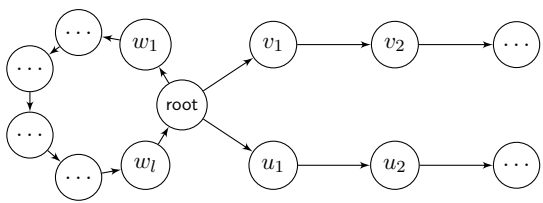
\begin{figure}[H]
\begin{tikzpicture}[scale=.8, transform shape]
\tikzset{vertex/.style = {shape=circle,draw, minimum size=1em}}
\tikzset{edge/.style = {->,> = latex'}}
% vertices
\node[vertex, scale=.8] (a) at  (0.5,0) {$\rootb$};
\node[vertex] (b) at  (2,-1) {$u_1$};
\node[vertex] (c) at  (2,1) {$v_1$};
\node[vertex] (d) at  (4, -1) {$u_2$};
\node[vertex] (e) at  (4, 1) {$v_2$};
\node[vertex] (f) at  (6, -1) {$\ldots$};
\node[vertex] (g) at  (6, 1) {$\ldots$};
%\node[vertex] (h) at  (8, -1) {$\ldots$};
\node[vertex] (j) at  (-2, -.5) {$\ldots$};
\node[vertex] (k) at  (-2, .6) {$\ldots$};
\node[vertex] (l) at  (-1.2, 1.2) {$\ldots$};
\node[vertex] (m) at  (-1.2, -1.2) {$\ldots$};
\node[vertex] (n) at  (-.1, 1) {$w_1$};
\node[vertex] (o) at  (-.1, -1) {$w_l$};

%\tikzset{vertex/.style = {shape=circle,draw, blue, minimum size=1em}}
%\node[vertex] (i) at  (8, 1) {$\ldots$};

%edges
\draw[edge] (a) to (c);
\draw[edge] (a) to (b);
\draw[edge] (c) to (e);
\draw[edge] (b) to (d);
\draw[edge] (e) to (g);
\draw[edge] (d) to (f);
%\draw[edge] (f) to (h);
%\draw[edge] (g) to (i);

\draw[edge] (a) to (n);
\draw[edge] (n) to (l);
\draw[edge] (l) to (k);
\draw[edge] (k) to (j);
\draw[edge] (j) to (m);
\draw[edge] (m) to (o);
\draw[edge] (o) to (a);

%\def\Radius{1.3}
%\draw (a) arc(0:360:\Radius) -- cycle;

%\draw[edge] (h) to (j);
%\draw[edge] (i) to (k);
%\draw[orange,rotate=45,shift={(3 cm,5 cm)}] \draw[edge] (e) to (g);
\end{tikzpicture}
\caption{A one-circuit directed graph with one branching vertex} \label{fig3}
\end{figure}
\noindent
We refer to $l$ as the {\it length of the circuit}.
For a subset $\{\lambda_w : w \in W\}$ of $\mathbb C$, consider the {\it weight system} $\lambdab : W \rar \mathbb C$ given by $\lambdab(w):=\lambda_w,$ $w \in W.$
The {\em weighted shift operator} $S_{\lambdab, \mathscr G}$ on ${\mathscr G}$
 is defined by
   \begin{align*}
   \begin{aligned}
{\mathscr D}(S_{\lambdab, \mathscr G}) & := \{f \in \ell^2(W) \colon
\varLambda_{\mathscr G} f \in \ell^2(W)\},
   \\
S_{\lambdab, \mathscr G} f & := \varLambda_{\mathscr G} f, \quad f \in {\mathscr
D}(S_{\lambdab, \mathscr G}),
   \end{aligned}
   \end{align*}
where $\varLambda_{\mathscr G}$ is the mapping defined on
complex functions $f$ on $V$ by
   \begin{align*}
(\varLambda_{\mathscr G} f) (w) :=\lambda_w \cdot f\big(\parent w\big), \quad w \in W.
   \end{align*}
\end{definition}
\begin{remark}
\label{rmk1}
Assume that $S_{\lambdab, \mathscr G}$ belongs to $\mathcal B(\ell^2(W)).$ It is well-known that the above notion is closely related to the notions of weighted shift on rooted directed trees and composition operator on discrete measure spaces  (refer to \cite{FSW}, \cite{JJS}, \cite[Theorem 3.2.1]{BJJS}). This is summarized as follows:
\begin{enumerate}
\item[(i)] Consider the weighted shift operator $S_{\lambdab, \mathscr T}$ with weight system $\lambdab|_{V \setminus \{\rootb\}}$ on the rooted directed tree $\mathscr T$. Note that $\ell^2(V \setminus \{\rootb\})$ is an invariant subspace for $S_{\lambdab, \mathscr G}$ and
%{\red \beqn
%S_{\lambdab, \mathscr G}|_{\ell^2(V)} =\begin{cases} S_{\lambdab, \mathscr T} + \lambda_{\rootb}e_{\rootb} \otimes e_{\rootb} & \mbox{if~}l=0, \\
%S_{\lambdab, \mathscr T} + \lambda_{w_1}e_{w_1} \otimes e_{\rootb} & \mbox{otherwise}.
%\end{cases}
%\eeqn
%}
\beqn
S_{\lambdab, \mathscr G}|_{\ell^2(V \setminus \{\rootb\})} &=& S_{\lambdab, \mathscr T}|_{\ell^2(V \setminus \{\rootb\})}, \\ 
  S_{\lambdab, \mathscr G}e_{\rootb} &=& \begin{cases} S_{\lambdab, \mathscr T}e_{\rootb} + \lambda_{\rootb}e_{\rootb}   & \mbox{if~}l=0, \\
S_{\lambdab, \mathscr T}e_{\rootb} + \lambda_{w_1}e_{w_1}   & \mbox{otherwise}.
\end{cases}
\eeqn
\item[(ii)] Consider the map $\phi : W \rar W$ given by $$\phi(w)=\parent{w}, \quad w \in W.$$ Then $S_{\lambdab, \mathscr G}$ can be identified with the weighted composition operator $C_{\phi, \lambdab}$ given by
\beqn
(C_{\phi, \lambdab}f)(w)=\lambda_w f(\phi(w)), \quad f \in \ell^2(W), ~w \in W.
\eeqn
\end{enumerate}
\end{remark}

Let $\mathscr G = (W, F)$ be a directed graph and let
$U$ be a subset of $W$. Set
$$\mathsf{Chi}_{\mathscr G}(U) := \bigcup_{u\in U} \{w\in W
\colon (u,w) \in F\}.$$ 
%Let $\mathbb N$ denote the set of non-negative integers.
We define inductively $\childm{n}{G}{U}$  for 
$n \in \mathbb N$ as follows: 
\beqn
\childm{n}{G}{U}:= 
\begin{cases} U & ~\mbox{if }~n=0, \\
\mathsf{Chi}_{\mathscr G}\big({\childm{n-1}{G}{U}}\big) & ~\mbox{if~} n \geqslant 
1. \end{cases}
\eeqn
Given $w \in W$, we set $\mathsf{Chi}_{\mathscr G}(w):=\mathsf{Chi}_{\mathscr G}(\{w\})$.

\begin{assumption} {\it Let $\mathscr G = (W, F)$ be a one-circuit directed graph associated with the rooted directed tree.
Unless stated otherwise, $\mathscr G$ is assumed to be countably infinite $($that is, $W$ is countably infinite$)$ and leafless $($that is, $\mbox{Chi}_{\mathscr G}(w) \neq \emptyset$ for every $w \in W).$ Further, we assume that $\lambdab$ consists of positive numbers and $S_{\lambdab, \mathscr G}$ belongs to $\mathcal B(\ell^2(W)).$}
\end{assumption}

In what follows, we need the following elementary fact in the sequel.
%(cf. \cite[Corollary 2.1.5]{JJS}).

\begin{lemma}
Let $l \in \mathbb N$ and let $\mathscr G = (W, F)$ be a one-circuit directed graph associated with the rooted directed tree $\mathscr T = (V, E)$ and the circuit $\{w_1, \ldots, w_{l+1}\},$ where $w_{l+1}:=\rootb.$
Then
 $W$ can be partitioned as follows:
\beq \label{partition}
W =  \left(\bigsqcup_{k=1}^{l+1} \childm{k}{G}{\rootb} \right) \bigsqcup \left(\bigsqcup_{k=1}^{l+1} \bigsqcup_{m=1}^{\infty} \childm{(l+1)m+k}{T}{\rootb} \right),
\eeq
where $\bigsqcup$ denotes the disjoint sum.
\end{lemma}
\begin{proof}
Clearly, the right hand side of (2.1) is contained in $W$. To see the inclusion other way round, first note that $W = V \sqcup \{w_1, \ldots, w_l\}$ and \beq
\label{rel} \childm{k}{G}{\rootb} = \childm{k}{T}{\rootb} \sqcup \{w_k\}, \quad k = 1, \ldots, l+1. \eeq
Let $w \in W$. Then either $w \in V \setminus \{\rootb\}$ or $w = w_j$ for some $j = 1, \ldots, l+1$.  If $w = w_j$ for some $j = 1, \ldots, l+1$, then, as noted above, $w \in \childm{j}{G}{\rootb}$. Let $w \in V \setminus \{\rootb\}$. Then by \cite[Corollary 2.1.5]{JJS}, there exists a positive integer $n$ such that $w \in \childm{n}{T}{\rootb}$. 
By division algorithm, there exists $m \in \mathbb N$ such that $n = m(l+1)+k$ for some $k=0, \ldots, l$. It is now easy to see that $w$ belongs to the right hand side of \eqref{partition}.
%If $m=0$, then, as noted above, $w \in \childm{n}{G}{\rootb} =  \childm{n}{T}{\rootb} \sqcup \{w_n\}$. The case $m \Ge 1$ is obvious. 
%This completes the proof. 
%{\red Since $\mathbb N = \{0\} \cup \{(l+1)m + k  : k =1, \ldots, l+1, m \in \mathbb N\}$, the desired conclusion follows immediately.}
\end{proof}

\section{Analyticity of weighted shifts on one-circuit directed graphs}

The main result of this section characterizes a class of left-invertible analytic weighted shifts on a locally finite one-circuit directed graph (see Theorem \ref{main-1} and Remark \ref{left-i}). 
Unlike weighted shifts on rooted directed trees (see \cite[Lemma 3.3]{CT}), these shifts need not be analytic. Indeed, they may admit non-zero eigenvalues (cf. \cite[Theorem 2.1]{C}). 

\begin{theorem}
\label{main-1}
Let $l \in \mathbb N$ and let $\mathscr G = (W, F)$ be a one-circuit directed graph associated with the locally finite rooted directed tree $\mathscr T = (V, E)$ with root $\rootb$ and the circuit $\{w_1, \ldots, w_{l+1}\},$ where $w_{l+1}:=\rootb.$  
Let $S_{\lambdab, \mathscr G}$ be a weighted shift on $\mathscr G$ such that $\inf \lambdab > 0$ and let 
\beq \label{lambda-k}
\lambdab^{(k)} := \lambdab (\lambdab \circ \mathsf{par}) \cdots (\lambdab \circ \mathsf{par}^{k-1}), \quad k \geqslant 1.
\eeq 
Then $S_{\lambdab, \mathscr G}$ is analytic if and only if for each $k=1, \ldots, l+1,$ the following series diverges:
\beq \label{series}
&& \sum_{v \, \in \, \childm{k}{G}{\rootb}} \left(\frac{\lambdab^{(k)}(v)}{\lambdab^{(k)}(w_k)}\right)^2  \notag \\ & + &
\sum_{m=1}^{\infty}  \ \sum_{v \, \in \, \childm{(l+1)m + k}{T}{\rootb}} \left(\frac{\lambdab^{((l+1)m + k)}(v)}{\lambdab^{((l+1)m + k)}(w_k)}\right)^2.
\eeq
\end{theorem}
\begin{remark} \label{left-i}
Note that the set $\{S_{\lambdab, \mathscr G}e_w\}_{w \in W}$ is orthogonal in $\ell^2(W)$. It is now easy to see that the assumption $\inf \lambdab >0$ implies that $S_{\lambdab, \mathscr G}$ is left-invertible, and hence the hyper-range of $S_{\lambdab, \mathscr G}$ is closed in $\ell^2(W)$. Finally, note that since $\mathscr G$ is locally finite, the first sum in \eqref{series} is finite.  
\end{remark}

In the proof of Theorem \ref{main-1}, we need a couple of lemmas. 
We begin with the following variant of a known fact pertaining to weighted composition operators (see \cite[Theorem 2.1.1]{K} and \cite[Remark 45]{BJJS-0}).
%Since its reference is not easily available, we include a proof for the sake of completeness.

\begin{lemma} \label{power}
Let $\mathscr G = (W, F)$ be a one-circuit directed graph associated with a locally finite rooted directed tree. Let $S_{\lambdab, \mathscr G}$ be a weighted shift on $\mathscr G$ such that $\inf \lambdab > 0$.  Then, for any positive integer $k$, the range of $S^k_{\lambdab, \mathscr G}$ is given by
\beqn
%\ran S^k_{\lambdab, \mathscr G}= 
\Big\{\!f \in \ell^2(W) : \frac{f}{\lambdab^{(k)}}\Big|_{\childm{k}{G}{w}} \mbox{is constant for each } w \in W \! \Big\},
\eeqn
where $\lambdab^{(k)}$ is given by \eqref{lambda-k}.
\end{lemma}

\begin{proof}
Suppose that $f \in \ran S^k_{\lambdab, \mathscr G}$. Thus there exists $g \in \ell^2(W)$ such that $f(w) = \lambdab^{(k)}(w) g(\mathsf{par}^k(w))$, $w \in W$. Let $v \in W$ be fixed and let $u, w \in \childm{k}{G}{v}$. Then 
\beqn \frac{f(u)}{\lambdab^{(k)}(u)} = g(\mathsf{par}^k(u)) = g(v) = g(\mathsf{par}^k(w)) = \frac{f(w)}{\lambdab^{(k)}(w)}. \eeqn
Thus $\frac{f}{\lambdab^{(k)}}|_{\childm{k}{G}{w}}$ is constant.
Conversely, suppose that $f \in \ell^2(W)$ is such that $\frac{f}{\lambdab^{(k)}}|_{\childm{k}{G}{w}}$ is constant for each $w \in W$. This allows us to define $g : W \rar \mathbb C$ by
\beqn
g(w) = \frac{f(v)}{\lambdab^{(k)}(v)}, \quad v \in \childm{k}{G}{w}.
\eeqn
Since $\inf \lambdab > 0$ and $f \in \ell^2(W)$,  $g \in \ell^2(W)$. Further, $$( S^k_{\lambdab, \mathscr G}g) (w)= \lambdab^{(k)}(w) g(\mathsf{par}^k(w)) = f(w), \quad w \in W.$$ Thus $f \in \ran S^k_{\lambdab, \mathscr G}$ and the proof is over.
%\beqn
%\ran S_{\lambdab, \mathscr G}= \Big\{f \in \ell^2(W) : \frac{f}{\lambdab}\Big|_{\child{w}} \mbox{ is constant for each } w \in W\Big\}.
%\eeqn 
%\begin{case}
%Arbitrary $k:$
%\end{case}
%Assume the conclusion of the lemma for $k$. Let $f \in \ran C^{k+1}_{\phi, \lambdab}$. Thus 
\end{proof}

We next analyze the hyper-range of weighted shifts on a one-circuit directed graph associated with the locally finite rooted directed tree.
\begin{lemma} \label{main}
Let $l \in \mathbb N$ and let $\mathscr G = (W, F)$ be a one-circuit directed graph associated with the locally finite rooted directed tree $\mathscr T = (V, E)$ with root $\rootb$ and the circuit $\{w_1, \ldots, w_{l+1}\},$ where $w_{l+1}:=\rootb.$  Let $S_{\lambdab, \mathscr G}$ be a weighted shift on $\mathscr G$ such that $\inf \lambdab > 0$ and let \beq \label{gk}
g_k &:=& \!\!\!\sum_{w \in \childm{k}{G}{\rootb}} \frac{\lambdab^{(k)}(w)}{\lambdab^{(k)}(w_k)} e_w  \notag  +   \sum_{m=1}^{\infty}  ~\sum_{w \in \childm{(l+1)m+k}{T}{\rootb}}\frac{\lambdab^{((l+1)m + k)}(w)}{\lambdab^{((l+1)m + k)}(w_k)} e_w, \\ && \hspace*{7.5cm} k=1, \ldots, l+1.  
\eeq 
Then the following statements hold:
\begin{enumerate}
\item[(i)] 
%The dimension of the hyper-range of $S_{\lambdab, \mathscr G}$ is at most $l+1$.  
The hyper-range of $S_{\lambdab, \mathscr G}$ is spanned by $h_k: W \rar [0, \infty)$, $k=1, \ldots, l+1,$ given by
\beqn
h_k := \begin{cases} g_k & \mbox{if the series in \eqref{series} converges}, \\
  0 & \mbox{otherwise}. 
  \end{cases}
  \eeqn 
\item[(ii)] The hyper-range of $S_{\lambdab, \mathscr G}$ is given by
\beqn
\Big\{\!f \in \ell^2(W) : \frac{f}{\lambdab^{(k)}}\Big|_{\childm{k}{G}{\rootb}} \mbox{is constant for each } k \geqslant 1 \! \Big\}.
\eeqn
\end{enumerate}
\end{lemma}
\begin{proof} Let $f$ belong to the hyper-range of $S_{\lambdab, \mathscr G}$. 
%\beqn
%W = \{w_1, \ldots, w_{l+1}\} \sqcup \left(\sqcup_{k=1}^{l+1} \sqcup_{m=0}^{\infty} \childn{(l+1)m+k}{\rootb}_{\mathscr T} \right)
%\eeqn
By \eqref{partition}, $f$ can be decomposed as 
\beq \label{decom}
f \,=\, \sum_{k=1}^{l+1} \sum_{w \in \childm{k}{G}{\rootb}}f(w)e_w ~+~ \sum_{k=1}^{l+1} \sum_{m=1}^{\infty} \sum_{w \in \childm{(l+1)m+k}{T}{\rootb}}f(w)e_w.
\eeq
Further, Lemma \ref{power} combined with the fact that $w_k \in \childm{(l+1)m+k}{G}{\rootb}$ yields the following identities:
\beqn
f(w) = f(w_k) \, \frac{\lambdab^{((l+1)m + k)}(w)}{\lambdab^{((l+1)m + k)}(w_k)}, \quad w \in \childm{(l+1)m+k}{G}{\rootb}, \\ m \in \mathbb N, ~~k=1, \ldots, l+1.
\eeqn
Substituting this into \eqref{decom}, we obtain $f= \sum_{k=1}^{l+1} f(w_k) g_k$, where $g_k,$ $k=1, \ldots, l+1,$ is given by
\eqref{gk}.
%Let $I:=\{k \in \{1, \ldots, l +1\} : f(w_k) \neq 0\}.$ Since $f \in \ell^2(W)$, $g_k \in \ell^2(W)$ for every $k \in I.$ This shows that 
Since $f \in \ell^2(W)$ and supports of $g_1, \ldots, g_{l+1}$ are disjoint, if $g_k \notin \ell^2(W)$ for some $k=1, \ldots, l+1$, then $f(w_k)=0$.  
%This shows that the dimension of the hyper-range of $S_{\lambdab, \mathscr G}$ is at most $l+1$. 
This completes the proof of (i). 

To see (ii), note that by Lemma \ref{power}, $f$ is in the hyper-range of $S_{\lambdab, \mathscr G}$ if and only if 
\beq\label{eq1}
\frac{f}{\lambdab^{(k)}}\Big|_{\childm{k}{G}{w}} \mbox{ is constant for each } w \in W \mbox{ and } k \geqslant 1.
\eeq
We claim that \eqref{eq1} is equivalent to 
\beq\label{eq2}
\frac{f}{\lambdab^{(k)}}\Big|_{\childm{k}{G}{\rootb}} \mbox{ is constant for each }  k \geqslant 1.
\eeq
Clearly, \eqref{eq1} implies \eqref{eq2} by taking $w = \rootb$. Suppose that \eqref{eq2} holds. Let $w \in W$. By \eqref{partition}, there exists $n \in \mathbb N$ such that $w \in \childm{n}{G}{\rootb}$. Now for any $k \in \mathbb N$, 
\beqn
\childm{k}{G}{w} \subseteq \childm{n+k}{G}{\rootb}.
\eeqn
Let $u, v \in \childm{k}{G}{w}$. Then
\beqn
\frac{f(u)}{\lambdab^{(k)}(u)} &\overset{\eqref{lambda-k}}=& 
%\frac{f(u)}{\lambda_u \cdots \lambda_{\mathsf{par}^{k-1}(u)}} \frac{\lambda_w \cdots \lambda_{\mathsf{par}^{n-1}(w)}}{\lambda_w \cdots \lambda_{\mathsf{par}^{n-1}(w)}}\\
%&=& 
\frac{f(u)}{\lambdab^{(n+k)}(u)} \big(\lambda_w \cdots \lambda_{\mathsf{par}^{n-1}(w)}\big) \\
&\overset{\eqref{eq2}}=& \frac{f(v)}{\lambdab^{(n+k)}(v)} (\lambda_w \cdots \lambda_{\mathsf{par}^{n-1}(w)}) \\ &=& \frac{f(v)}{\lambdab^{(k)}(v)}.
\eeqn
This establishes \eqref{eq1} and hence completes the proof.
\end{proof}

%We next discuss a partial converse to Proposition \ref{main}(ii), which is crucial for the construction carried out in the subsequent section.

We now complete the proof of Theorem \ref{main-1}.
\begin{proof}[Proof of Theorem \ref{main-1}] If, for each $k=1, \ldots, l+1,$ the series in \eqref{series} diverges,  then by Lemma \ref{main}(i), $S_{\lambdab, \mathscr G}$ is analytic. 
Conversely, suppose that for some $k=1, \ldots, l+1$, the series in \eqref{series} converges. In particular, $g_k$, as given by \eqref{gk}, belongs to $\ell^2(W)$. 
We check that $g_k$ belongs to the hyper-range of $S_{\lambdab, \mathscr G}$.
In view of Lemma \ref{main}(ii), we only need to verify that $f:=g_k$ satisfies \eqref{eq2}. 
To see that, let $n \geqslant 1$ be an integer and write $n=(l+1)p+r$ for some $p \in \mathbb N$ and $0 \leqslant r < l+1$. It now follows from \eqref{rel} that 
\beqn
 \childm{n}{G}{\rootb} = \begin{cases} \childm{n}{T}{\rootb} \sqcup \{w_r\} & \mbox{if~}r >0, \\
 \childm{n}{T}{\rootb} \sqcup \{\rootb\} & \mbox{otherwise}.
 \end{cases}
 \eeqn
Thus by \eqref{gk}, we have \beqn \frac{g_k}{\lambdab^{(n)}}\Big|_{\childm{n}{G}{\rootb}} = \begin{cases}   \frac{1}{\lambdab^{(n)}(w_k)} & \mbox{if~}r=k, \\
\frac{1}{\lambdab^{(n)}(\rootb)} & \mbox{if~}r=0~\mbox{and~}k=l+1,\\ 0 & \mbox{otherwise}. 
\end{cases}
\eeqn
This completes the proof of the theorem.
\end{proof}

\section{Analytic $3$-isometries without wandering subspace property}

We now construct a family of cyclic analytic $3$-isometries, which do not admit the wandering subspace property and which are norm-increasing on the orthogonal complement of a one-dimensional space.

\begin{example} \label{construction}
Consider the directed tree $\mathscr T$ with the set of vertices 
$V:=\mathbb{N}$ and $\mathsf{root}=0$. We further require that $\mathsf{Chi}_{\mathscr T}(n)=\{n+1\}$ for
all 
$n \in V$ (see \cite[Equation (6.2.10)]{JJS}). 
Let $\mathscr G = (W, F)$ be the following one-circuit directed graph associated with the rooted directed tree $\mathscr T$:
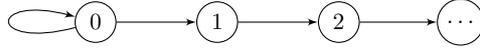
\begin{figure}[H]
\begin{tikzpicture}[scale=.8, transform shape]
\tikzset{vertex/.style = {shape=circle,draw, minimum size=1em}}
\tikzset{every loop/.style = {min distance = 15mm,  looseness = 15}}
\tikzset{edge/.style = {->,> = latex'}}
% vertices
\node[vertex] (a) at  (0,0) {$0$};
%\node[vertex] (b) at  (2,-1) {$1$};
\node[vertex] (c) at  (2,0) {$1$};
%\node[vertex] (d) at  (4, -1) {$3$};
\node[vertex] (e) at  (4, 0) {$2$};
%\node[vertex] (f) at  (6, -1) {$\ldots$};
\node[vertex] (g) at  (6, 0) {$\ldots$};
%\node[vertex] (h) at  (8, -1) {$\ldots$};

%\tikzset{vertex/.style = {shape=circle,draw, blue, minimum size=1em}}
%\node[vertex] (i) at  (8, 1) {$\ldots$};

%edges
\draw[edge] (a) to (c);
%\draw[edge] (a) to (c);
\draw[edge] (c) to (e);
%\draw[edge] (c) to (e);
\draw[edge] (e) to (g);
%\draw[edge] (e) to (g);
%\draw[edge] (f) to (h);
%\draw[edge] (g) to (i);

%\draw[edge] (a) to (n);
%\draw[edge] (n) to (l);
%\draw[edge] (l) to (k);
%\draw[edge] (k) to (j);
%\draw[edge] (j) to (m);
%\draw[edge] (m) to (o);
\draw[loop left][edge] (a) to (a);

%\def\Radius{.5}
%\draw (a) arc(0:360:\Radius) -- cycle;

%\draw[edge] (h) to (j);
%\draw[edge] (i) to (k);
%\draw[orange,rotate=45,shift={(3 cm,5 cm)}] \draw[edge] (e) to (g);
\end{tikzpicture}
\end{figure}
\noindent
For positive numbers $a, b$, consider the degree $2$ polynomial $p_{a, b} : \mathbb N \rar (0, \infty)$ given by $p_{a, b}(n) = 1+an + bn^2$, $n \in \mathbb N.$ 
By the Archimedean property, there exists $\epsilon_0 \in (0, 1)$ such that $\epsilon + b(2/\epsilon -1) > a$ for every $\epsilon \in (0, \epsilon_0)$.
For $\epsilon \in (0, \epsilon_0)$, let $\lambdab_{\epsilon}=\{\lambda_w : w \in W\}=\{\lambda_n : n \in \N\}$ denote the weight system of positive real numbers given by
\beq
\label{wts}
\left.
 \begin{array}{ccc}
\lambda_{0}:=\sqrt{1-\epsilon}, \quad
\lambda_1 : =  \sqrt{\frac{\epsilon^{3}}{K_{\epsilon}}} \ \  \mbox{with} \ \ K_{ \epsilon}:= \epsilon^2 - \epsilon(a+b) + 2b > 0, \vspace{.2cm} \\ 
\lambda_n = \sqrt{\frac{p_{a, b}(n-1)}{p_{a, b}(n-2)}}, \quad n \geqslant 2.
\end{array}
\right\}
\eeq
%where $K_{ \epsilon}:= \epsilon^2 - \epsilon(a+b) + b.$
Let $S_{\lambdab_{\epsilon}, \mathscr G}$ (resp. $S_{\lambdab_{\epsilon}, \mathscr T}$) be a weighted shift on $\mathscr G$ (resp. $\mathscr T$). An inductive argument shows that $e_w$ belongs to the linear span of $\{S^n_{\lambdab_{\epsilon}, \mathscr G}e_0 : n \in \mathbb N\}$ for every $w \in W$, and hence
$S_{\lambdab_{\epsilon}, \mathscr G}$ is cyclic with cyclic vector $e_0$:
\beqn
\bigvee \{S^n_{\lambdab_{\epsilon}, \mathscr G}e_0 : n \in \mathbb N\} = \ell^2(W).
\eeqn
%Since $S_{\lambdab_{\epsilon}, \mathscr T} \in \mathcal B(\ell^2(V))$ is left-invertible, 
By \eqref{wts} and Remark \ref{left-i}, $S_{\lambdab_{\epsilon}, \mathscr G} \in \mathcal B(\ell^2(W))$ is left-invertible. 
Note that Cauchy dual operator $S'_{\lambdab_{\epsilon}, \mathscr G}$ is the weighted shift $S_{\lambdab'_{\epsilon}, \mathscr G}$ on $\mathscr G$ with weight system $\lambdab'_{\epsilon}=\{\lambda'_w : w \in W\}$ given by 
\beq
\label{dual-wts}
\lambda'_{0} = \frac{\lambda_{0}}{\lambda^2_{0} + \lambda^2_1}, ~~ 
\lambda'_{1} = \frac{\lambda_1}{\lambda^2_{0} + \lambda^2_1}, ~~
\lambda'_n = \sqrt{\frac{p_{a, b}(n-2)}{p_{a, b}(n-1)}}, \quad n \geqslant 2. 
\eeq
{\it
Then we have the following statements:
\begin{enumerate}
\item  $S_{\lambdab_{\epsilon}, \mathscr G}$ is an analytic $3$-isometry.
\item $S_{\lambdab_{\epsilon}, \mathscr G}$ is norm-increasing if and only if $K_{ \epsilon} \leqslant \epsilon^2$.
\item $S_{\lambdab_{\epsilon}, \mathscr G}$ admits the wandering subspace property if and only if \beq \label{wsp-epsi} K_{ \epsilon} < \frac{\epsilon^3}{ \sqrt{1-\epsilon}(1-\sqrt{1-\epsilon})}. \eeq
\item $S'_{\lambdab_{\epsilon}, \mathscr G}$ admits Wold-type decomposition if and only if \eqref{wsp-epsi} holds.
\end{enumerate}
}
We verify the above statements as follows.
Since $\{p_{a, b}(n)\}_{n \in \mathbb N}$ is an increasing sequence, $S_{\lambdab_{\epsilon}, \mathscr T}$ is a norm-increasing weighted shift on $\mathscr T$.
Further, since $\|S^k_{\lambdab_{\epsilon}, \mathscr T}e_1\|^2 = p_{a, b}(k),$ $k \in \mathbb N,$ is a degree $2$ polynomial,  
by \cite[Theorem 2.1]{AL}, $S_{\lambdab_{\epsilon}, \mathscr T}$ is a $3$-isometry. 
Since for any positive integer $k,$ the sequence $\{S^k_{\lambdab_{\epsilon}, \mathscr G}e_w\}_{w \in W}$ is orthogonal, $S_{\lambdab_{\epsilon}, \mathscr G}$ is a norm-increasing $3$-isometry if and only if
\beqn
 \|S_{\lambdab_\epsilon, \mathscr G}e_{0}\|  ~\geqslant ~ 1, \hspace{3.5cm} \\
 1- 3\|S_{\lambdab_\epsilon, \mathscr G}e_{0}\|^2 + 3\|S^2_{\lambdab_\epsilon, \mathscr G}e_{0}\|^2 - \|S^3_{\lambdab_\epsilon, \mathscr G}e_{0}\|^2 ~=~0.
\eeqn
In view of \eqref{wts}, $\|S_{\lambdab_{\epsilon}, \mathscr G}e_{0}\|  \geqslant  1$ if and only if $K_{ \epsilon} \leqslant \epsilon^2.$ It is thus  
easy to see that the above identities simplify to
\beq
%\label{3-iso}
  K_{ \epsilon} \leqslant \epsilon^2,  \hspace{3.5cm}  \notag  \\  
 \left. 
 \begin{array}{ccc} 
\label{3-iso} 
  1 - 3\lambda^2_{1} +3\lambda^2_{1}\lambda^2_{2} - \lambda^2_{1}\lambda^2_{2}\lambda^2_{3}  \vspace{.2cm}  \\    = ~~\lambda^2_{0}(3 - 3 \lambda^2_{0} - 3\lambda^2_{1} + \lambda^4_{0} +  \lambda^2_{0}\lambda^2_1 
 +\lambda^2_{1} \lambda^2_2).
 \end{array}
\right\}
\eeq
It is easy to see that \eqref{3-iso} is immediate from \eqref{wts}.
Thus $S_{\lambdab_{\epsilon}, \mathscr G}$ is a $3$-isometry, which is norm-increasing if and only if $K_{ \epsilon} \leqslant \epsilon^2.$ 
On the other hand, by Theorem \ref{main-1}, $S_{\lambdab_{\epsilon}, \mathscr G}$ is analytic if and only if 
\beqn 
 1 + \left(\frac{\lambdab^{(1)}_\epsilon(1)}{\lambdab^{(1)}_\epsilon(0)}\right)^2 + 
\sum_{m=1}^{\infty}  \  \left(\frac{\lambdab^{(m + 1)}_\epsilon(m+1)}{\lambdab^{(m + 1)}_\epsilon(0)}\right)^2 = \infty.
\eeqn
%where $\lambdab^{(k)}_{\epsilon}$ is given by \eqref{lambda-k}.
%\beqn 
%\sum_{m=0}^{\infty}  \  \left(\frac{\lambdab^{(m + 1)}(m + 1)}{\lambdab^{(m + 1)}(0)}\right)^2 = \infty.
%\eeqn
A routine inductive argument using \eqref{lambda-k} shows however that for $m \in \mathbb N,$
\beqn 
\lambdab^{(m + 1)}_\epsilon(m + 1) &=& \prod_{j=1}^{m + 1} \lambda_j   \overset{\eqref{wts}} = \lambda_1 \, p_{a, b}(m),\\
\lambdab^{(m + 1)}_\epsilon(0) &=&   \lambda^{m+1}_{0}.
\eeqn
Since $p_{a, b}(m) \geqslant 1$ for every $m \in \mathbb N$, we have
\beqn 
\sum_{m=0}^{\infty}  \  \left(\frac{\lambdab^{(m + 1)}_\epsilon(m + 1)}{\lambdab^{(m + 1)}_\epsilon(0)}\right)^2  ~=~   \sum_{m=0}^{\infty}   \frac{\lambda_1 \, p_{a, b}(m)}{\lambda^{2(m+1)}_{0}}  ~\geqslant ~ \lambda_1 \sum_{m=0}^{\infty}   \frac{1}{\lambda^{2(m+1)}_{0}}.
\eeqn
Also, since $\lambda_{0} <1$,  the series on the right hand side diverges, and hence
$S_{\lambdab_{\epsilon}, \mathscr G}$ is analytic. This completes verifications of (i) and (ii).
The above argument applied to the Cauchy dual operator $S_{\lambdab'_{\epsilon}, \mathscr G}$  together with \eqref{dual-wts} yields that $S_{\lambdab'_{\epsilon}, \mathscr G}$ is analytic if and only if the following series diverges:
\beqn 
\sum_{m=0}^{\infty}  \  \left(\frac{\lambdab'^{(m + 1)}_\epsilon(m + 1)}{\lambdab'^{(m + 1)}_\epsilon(0)}\right)^2 = \sum_{m=0}^{\infty}  \frac{(\lambda^2_{0} + \lambda^2_1)^{2(m+1)}}{\lambda^{2(m+1)}_{0}} \frac{\lambda'_1}{ p_{a, b}(m)}.
\eeqn
Since $p_{a, b}$ is a degree $2$ polynomial with positive coefficients, the above series diverges if and only if $\lambda^2_{0} + \lambda^2_1 > \lambda_0$.
On the other hand,  by \eqref{wts}, we obtain \beqn \lambda^2_{0} + \lambda^2_1 > \lambda_0 ~ \Longleftrightarrow ~ 1-\epsilon +  \frac{\epsilon^3}{K_{ \epsilon}} > \sqrt{1-\epsilon}  ~ \Longleftrightarrow ~
%K_{ \epsilon} < \frac{\epsilon^3}{\sqrt{1-\epsilon}(1-\sqrt{1-\epsilon})}.
\eqref{wsp-epsi} ~\mbox{holds}.
\eeqn
%Thus if $6b -2 a \geqslant \sqrt{2}$,
%then $S_{\lambdab'_{\epsilon}, \mathscr G}$ is not analytic, and hence 
Thus $S_{\lambdab'_{\epsilon}, \mathscr G}$ is analytic if and only \eqref{wsp-epsi} holds.
However, by \cite[Proposition 3.4]{S}, $S_{\lambdab_{\epsilon}, \mathscr G}$  admits the wandering subspace property if and only if $S_{\lambdab'_{\epsilon}, \mathscr G}$ is analytic. This yields (iii).
To see (iv), 
recall that $S_{\lambdab_{\epsilon}, \mathscr G}$ admits Wold-type decomposition if and only if $S'_{\lambdab_{\epsilon}, \mathscr G}$ admits Wold-type decomposition (see \cite[Proposition 3.4]{S}). Further, by (i) and (iii),  $S_{\lambdab_{\epsilon}, \mathscr G}$ admits Wold-type decomposition if and only if \eqref{wsp-epsi} holds. Combining last two observations, we obtain (iv).
It is interesting to note that in case \eqref{wsp-epsi} does not hold, then the hyper-range of $S_{\lambdab'_{\epsilon}, \mathscr G}$ is spanned by \beqn g=e_0 + \sum_{m=1}^{\infty} \frac{\lambdab'^{(m)}_{\epsilon}(m)}{\lambdab'^{(m)}_{\epsilon}(0)}e_m. \eeqn In particular, the space spanned by $g$ is not invariant under $S^*_{\lambdab'_{\epsilon}, \mathscr G}$ (since $\lambdab'_{\epsilon}$ is not eventually constant in view of \eqref{dual-wts}). 
%Note that (iv) is an immediate consequence of \cite[Proposition 3.4]{S}, (i) and (iii). %Alternatively, this may be deduced from the fact that the hyper-range of $S_{\lambdab'_{\epsilon}, \mathscr G}$, spanned by \beqn g=e_0 + \sum_{m=1}^{\infty} \frac{\lambdab'^{(m)}_{\epsilon}(m)}{\lambdab'^{(m)}_{\epsilon}(0)}e_m, \eeqn is not invariant under $S^*_{\lambdab'_{\epsilon}, \mathscr G}$ (since $\lambdab'_{\epsilon}$ is not eventually constant in view of \eqref{dual-wts}). 
%This completes verifications of (i)-(iv). 

Let us discuss some consequences of (i)-(iv). Firstly,
since $0 < \epsilon < 1,$ we obtain $\epsilon^2 \leqslant \frac{\epsilon^3}{ \sqrt{1-\epsilon}(1-\sqrt{1-\epsilon})},$
%$$1 \leqslant \frac{\epsilon}{ \sqrt{1-\epsilon}(1-\sqrt{1-\epsilon})} ~~\mbox{or~}~\epsilon^2 \leqslant \frac{\epsilon^3}{ \sqrt{1-\epsilon}(1-\sqrt{1-\epsilon})},$$ 
and hence it follows from (ii) and (iii) that if $S_{\lambdab_{\epsilon}, \mathscr G}$ is norm-increasing, then $S_{\lambdab_{\epsilon}, \mathscr G}$ necessarily  admits the wandering subspace property. Secondly,
since \beqn \lim_{\epsilon \rar 0^+}K_{\epsilon} = 2b > 0, \quad \lim_{\epsilon \rar 0^+} \frac{\epsilon^3}{ \sqrt{1-\epsilon}(1-\sqrt{1-\epsilon})}=0, \eeqn
we may conclude from (i) and (iii) that there exists $\epsilon_1 \in (0, \epsilon_0)$ such that $S_{\lambdab_{\epsilon}, \mathscr G}$ is an analytic $3$-isometry without the wandering subspace property
for every $\epsilon \in (0, \epsilon_1)$. It is worth noting that $S_{\lambdab_{\epsilon}, \mathscr G}$ is always norm-increasing on the orthogonal complement of the space spanned by $e_{0}$. Although $S_{\lambdab_{\epsilon}, \mathscr G}$ is not norm-increasing, we have
$$\|S_{\lambdab_{\epsilon}, \mathscr G}e_0\|^2 = 1-\epsilon +  \frac{\epsilon^3}{K_{ \epsilon}} \rar 1^{-}~ \mbox{as} ~\epsilon \rar 0^+.$$ 
In particular, Question \ref{Q} has a negative answer if we replace the assumption $\|S_{\lambdab_{\epsilon}, \mathscr G}e_0\| \geqslant 1$ by $\|S_{\lambdab_{\epsilon}, \mathscr G}e_0\| \geqslant 1-\delta$ for arbitrarily small $\delta > 0$.
\eop
\end{example}
\begin{remark}
Let $l \in \mathbb N$ and let $\mathscr G=(W, F)$ be a  one-circuit directed graph associated with the rooted directed tree (as in the preceding example) and the circuit $\{w_1, \ldots, w_{l+1}\}$, where $w_{l+1}:=\rootb$ is the only branching vertex of $\mathscr G$.
It is not difficult to see that there are no norm-increasing analytic $3$-isometry weighted shifts $S_{\lambdab, \mathscr G}$ without the wandering subspace property. 
Indeed, if 
\beqn
\lambda_{w_j} \geqslant 1, \quad j=2, \ldots, l+1, \quad \lambda^2_{w_{1}} + \lambda^2_1 \geqslant 1, \quad \frac{\lambda_{w_{1}}}{\lambda^2_{w_{1}} + \lambda^2_1} > \prod_{j=2}^{l+1} \lambda_{w_j}
\eeqn
(with the convention that product over an empty set is $1$),
then $\lambda_{w_{1}} \geqslant 1$, and hence $\lambda_{w_{1}} \geq \lambda^2_{w_{1}} + \lambda^2_1,$ which is not possible.
\end{remark}

\section{A dichotomy}

As noted above, if $S_{\lambdab, \mathscr G}$ is a norm-increasing analytic $3$-isometry weighted shift on a one-circuit directed graph associated with a rooted directed tree having one branching vertex, then it must admit the wandering subspace property.
In view of this, it is tempting to explore  
the wandering subspace problem for the class of weighted shifts on a one-circuit directed graph associated with a rooted directed tree having more than one branching vertex. However, it turns out that by incorporating more than one branching vertex in the underlying rooted directed tree does not yield an answer to Shimorin's question. Indeed, we have the following general fact.

\begin{lemma}
Let $T \in \mathcal B(\mathcal H)$ be an analytic norm-increasing operator with spectrum contained in the closed unit disc and let $[\ker T^*]_T$ be the $T$-invariant closed subspace generated by $\ker T^* ~(see ~\eqref{WS})$. Then $\mathcal H \ominus [\ker T^*]_T$ is either zero or infinite-dimensional.
\end{lemma}
\begin{proof}  Assume that $\mathcal H \ominus [\ker T^*]_T$ is finite-dimensional. 
Note that the hyper-range $\mathcal H'_u$  of $T'$ is a closed invariant subspace for $T'$ and let $S:=T'|_{\mathcal H'_u}.$ 
By \cite[Proposition 3.4]{S}, $\mathcal H'_u=\mathcal H \ominus [\ker T^*]_T,$ which is finite-dimensional. It now suffices to check that $\mathcal H'_u=\{0\}$. Assume that $\mathcal H'_u$ is non-zero. Thus $S$ is an invertible finite-dimensional operator on $\mathcal H'_u$, and hence there exists a non-zero $g \in \mathcal H'_u$ such that $Sg=\lambda g$ for some $\lambda \in \mathbb C \setminus \{0\}$. It follows that $T'g=\lambda g.$ 
However, $T'$ is a contraction (since $T$ is an expansion), and hence $|\lambda| \leqslant 1.$ As $T^*T'=I$, we must have $\lambda T^*g = g$.
Since the spectrum of $T$ is contained in the closed unit disc, $|\lambda|=1$. 
 Recall the fact that the fixed points of a contraction and its adjoint are same (see \cite[Proposition 3.1]{SF}). Applying this fact to the contractive operator $\overline{\lambda}T'$ shows that $T'^*g=\overline{\lambda}g$. It follows that  
$T'^* T' g = g$. Now observe that
\beqn
Tg = T'(T'^* T')^{-1} g = T'g = \lambda g.
\eeqn
It is immediate that $g$ belongs to the hyper-range of $T$, which is  $\{0\}$ by the analyticity of $T$. This contradicts the fact that $g$ is non-zero. Hence $\mathcal H'_u$ must be zero.
\end{proof}
\begin{remark}
An examination of Example \ref{construction} shows that one may construct analytic norm-increasing weighted shifts by letting constant weights with value bigger than $1$ for which $\mathcal H \ominus [\ker T^*]_T$ is non-zero and finite dimensional. The above lemma is not applicable in this case, since these operators do not have spectrum contained in the closed unit disc.
\end{remark}

The previous lemma together with Remark \ref{ap-sp} yields the following:
\begin{theorem}
Let $T \in \mathcal B(\mathcal H)$ be an analytic norm-increasing $m$-concave operator and let $[\ker T^*]_T$ be the $T$-invariant closed subspace generated by $\ker T^* ~(see ~\eqref{WS})$. Then either $T$ admits the wandering subspace property or $\mathcal H \ominus [\ker T^*]_T$ is of infinite dimension.
\end{theorem}

Combining the last theorem with Lemma \ref{main}(i), we obtain the following:
\begin{corollary}
Let $\mathscr G = (W, F)$ be a one-circuit directed graph associated with the locally finite rooted directed tree $\mathscr T = (V, E)$ with circuit of length $0.$   
Let $S_{\lambdab, \mathscr G}$ be a weighted shift on $\mathscr G$ such that $\inf \lambdab > 0$. If $S_{\lambdab, \mathscr G}$
is an analytic norm-increasing $m$-concave operator, then $S_{\lambdab, \mathscr G}$ admits the wandering subspace property. 
\end{corollary}

As a future direction, it is natural to incorporate more than one circuit in a rooted directed tree. In this case, the associated weighted shifts are no more composition operators on discrete measure spaces (note that the proof of Lemma \ref{power} relies heavily on the fact that these weighted shifts are composition operators). The problem of characterizing analytic weighted shifts within this class, an important step in the above problem, is of course of independent interest.

\medskip \textit{Acknowledgment}. \
The authors convey their sincere thanks to Zenon Jab{\l}o\'nski and Jan Stochel for several useful comments.

{}

\end{document}